\documentclass{amsart}
\usepackage[utf8]{inputenc}
\usepackage{amssymb,latexsym}
\usepackage{amsmath}
\usepackage{graphicx}
\usepackage{textcomp}
\usepackage[dvipsnames]{xcolor}
\usepackage{tikz}

\usepackage{amsthm,amssymb,enumerate,graphicx, tikz}
\usepackage{amscd}
\usepackage{setspace}
\usepackage{comment}
\usepackage{hyperref}
\usepackage{cleveref}
\usepackage{booktabs,tabularx}
\usepackage{booktabs}
\usepackage{float}
\usepackage{array} 
\newcolumntype{L}{>{\raggedright\arraybackslash}X} 

\def\@logofont{\footnotesize}
\def\@setaddresses{\par
  \nobreak \begingroup
  \footnotesize
  \def\author##1{\nobreak\addvspace\bigskipamount}%
  \def\\{\par\nobreak}%
  \interlinepenalty\@M
  \def\address##1##2{\begingroup
    \par\addvspace\bigskipamount\indent
    \@ifnotempty{##1}{(\ignorespaces##1\unskip) }%
    {\scshape\ignorespaces##2}\par\endgroup}%
  \def\curraddr##1##2{\begingroup
    \@ifnotempty{##2}{\nobreak\indent\curraddrname
      \@ifnotempty{##1}{, \ignorespaces##1\unskip}\/:\space
      ##2\par}\endgroup}%
  \def\email##1##2{\begingroup
    \@ifnotempty{##2}{\nobreak\indent\emailaddrname
      \@ifnotempty{##1}{, \ignorespaces##1\unskip}\/:\space
      \ttfamily##2\par}\endgroup}%
  \def\urladdr##1##2{\begingroup
    \def~{\char`\~}%
    \@ifnotempty{##2}{\nobreak\indent\urladdrname
      \@ifnotempty{##1}{, \ignorespaces##1\unskip}\/:\space
      \ttfamily##2\par}\endgroup}%
  \addresses
  \endgroup
}
\renewcommand*\subjclass[2][2010]{%
  \def\@subjclass{#2}%
  \@ifundefined{subjclassname@#1}{%
    \ClassWarning{\@classname}{Unknown edition (#1) of Mathematics
      Subject Classification; using '2000'.}%
  }{%
    \@xp\let\@xp\subjclassname\csname subjclassname@#1\endcsname
  }%
}

\usepackage{graphicx,amsmath,amssymb}
\usepackage{algorithm}
\usepackage{mathdots, comment}
\usepackage[noend]{algpseudocode}

\newtheorem{theorem}{Theorem}[section]

\newtheorem*{theorem*}{Theorem}
\newtheorem{proposition}[theorem]{Proposition}

\newtheorem{corollary}[theorem]{Corollary}

\theoremstyle{definition}
\newtheorem{definition}[theorem]{Definition}

\theoremstyle{remark}
\newtheorem{remark}[theorem]{Remark}
\newtheorem{example}[theorem]{Example}

\begin{document}
\title[Decomposition theorems for unmatchable pairs]{Decomposition theorems for unmatchable pairs in groups and field extensions}

\author[M. Aliabadi, J. Losonczy]{Mohsen Aliabadi$^{1}$ \and Jozsef Losonczy$^{2,*}$}
\thanks{$^1$Department of Mathematics, Clayton State University, 
2000 Clayton State Boulevard, Lake City, Georgia 30260, USA.  \url{maliabadi@clayton.edu}.\\
$^2$Department of Mathematics, Long Island University,
720 Northern Blvd, Brookville, New York 11548, USA. \url{Jozsef.Losonczy@liu.edu}.}
\thanks{$^*$Corresponding Author.}


\begin{abstract}
A theory of matchings for finite subsets of an abelian group, introduced in connection with a conjecture of Wakeford on canonical forms for homogeneous polynomials, has since been extended to the setting of field extensions and to that of matroids. Earlier approaches have produced numerous criteria for matchability and unmatchability, but have offered little structural insight. In this paper, we develop parallel structure theorems which characterize unmatchable pairs in both abelian groups and field extensions. Our framework reveals analogous obstructions to matchability: nearly periodic decompositions of sets in the group setting correspond to decompositions of subspaces involving translates of a subfield in the linear setting. This perspective not only recovers previously known results, but also leads to new matching criteria and guarantees the existence of nontrivial unmatchable pairs.
\end{abstract}

\maketitle

\section{Introduction}\label{intro}
\textbf{Overview.}
Let $G$ be an abelian group with operation written multiplicatively, and let \(A,B\subseteq G\) be nonempty finite sets. A \emph{matching} from $A$ to $B$ is a bijection \(f:A\to B\) such that \(af(a)\notin A\) for all \(a\in A\).  This notion was introduced in~\cite{Losonczy 1} in connection with a conjecture of Wakeford concerning the possible sets of monomials that can be removed from a generic homogeneous polynomial through a linear change in variables~\cite{Wakeford}. 

More specifically, to any set $A$ of same-degree monomials in the variables \(x_1, \ldots ,x_n\) with \(|A| = n(n-1)\), one associates a weighted bipartite graph $\mathcal{G}_A$. It was shown in~\cite{Losonczy 1} that the monomials in $A$ are generically removable if and only if the biadjacency matrix of $\mathcal{G}_A$ has nonzero determinant. The nonzero terms in the Leibniz expansion of this determinant correspond to certain matchings \( A  \to \nolinebreak B \), where the monomials in $A$ are identified with their exponent vectors and $B$ is a set of vectors corresponding to the operators \(x_i\frac{\partial}{\partial x_j}\) for $i\neq j$. It turns out that when each monomial in $A$ involves all of the variables, one can find a matching corresponding to a determinant term that cannot be canceled. In this way, such a matching provides a combinatorial certificate for removability and hence a mechanism for studying canonical forms for homogeneous polynomials.

Further work on matchings has included some general results in abelian groups~\cite{Losonczy 2, Aliabadi 0}, an extension to arbitrary groups~\cite{Eliahou 1}, a lower bound on the number of matchings via Hamidoune's isoperimetric method~\cite{Hamidoune}, a linear formulation over field extensions~\cite{Eliahou 2}, and a matroidal analogue~\cite{Aliabadi 3}. 

Recently, in \cite{Aliabadi 5}, the present authors established characterization theorems for matchable subsets in abelian groups and matchable subspaces over field extensions using Dyson transforms.  These results removed a previous reliance on a diverse collection of number-theoretic inequalities and their linear analogues. We show here that the two main characterization theorems in \cite{Aliabadi 5} can be used to derive structural decomposition theorems for unmatchable pairs.  Exploiting these decompositions, we obtain existence criteria for nontrivial unmatchable pairs and further consequences, both for finite subsets of abelian groups and for finite-dimensional subspaces over field extensions.

\medskip

\textbf{Organization of the paper.} In Sections~\ref{prelims for groups} and \ref{prelims for field extensions}, we recall background information on matchings in abelian groups and field extensions. Section \ref{Structure section groups} establishes a structure theorem for unmatchable pairs \( (A,B) \) of finite subsets of any abelian group (Theorem \ref{Structure groups}), exhibiting a nearly periodic decomposition for such pairs, and then develops consequences, e.g., inheritance of unmatchability to suitable quotients and a Chowla-type criterion. Section \ref{Linear structure section} presents a linear analogue for field extensions: any unmatchable pair of subspaces decomposes as a direct sum built in part from translates of a subfield (Theorem \ref{linearStruc}). Several corollaries follow, including results on Chowla subspaces and finite fields. Section \ref{future} describes a few directions for further work.

\subsection{Matchings in abelian groups.} \label{prelims for groups}
We begin with some terminology and notation. For any positive integer \( n \), we write \( [n] \) for \( \{1, \ldots, n\} \).  Let \( G \) be an abelian group and let \( A, B \) be nonempty finite subsets of \( G \) such that \( |A| = |B| \). A bijective mapping \( f : A \rightarrow B \) is called a {\em matching} from \( A \) to \( B \) if \( af(a) \notin A \) for all \( a \in A \). If there exists at least one matching \( A \rightarrow B \), we say that the pair $(A,B)$ is \emph{matchable}; otherwise, we call $(A,B)$ \emph{unmatchable}.

In order for \( (A,B) \) to be matchable, it is clearly necessary that \( 1\notin B \). We say that the group $G$ has the \emph{matching property} if this condition is also sufficient. The following results were established in the abelian group setting in \cite{Losonczy 2}:
\begin{itemize}
    \item For a finite nonempty subset $A$ of $G$, the pair \( (A,A) \) is matchable if and only if \( 1\notin A \).
    \item $G$ has the matching property if and only if $G$ is torsion-free or cyclic of prime order.
\end{itemize}

Our main objective in the first part of this paper will be to provide a structural decomposition of unmatchable sets in \( G \) (Theorem \ref{Structure groups}).  Several applications of this result will then be presented.

\subsection{Matching subspaces in a field extension.} \label{prelims for field extensions}
For any subset \( S \) of a vector space \( V \), we let \( \langle S \rangle \) denote the subspace of \( V \) spanned by \( S \). If \( S = \{x_1, \ldots, x_n\} \), we also write \( \langle x_1, \ldots, x_n \rangle \) for this subspace. Given a field extension \( K \subseteq L \) and \( K \)-subspaces \( A \) and \( B \) of \( L \), we use \( AB \) to denote the \emph{Minkowski product} of \( A \) and \( B \):
\[
AB = \{ ab : a \in A,\ b \in B \}.
\]

In \cite{Eliahou 2}, Eliahou and Lecouvey introduced a notion of matchability for subspaces in a field extension. We give their definition below.

Let \( K \subsetneq L \) be a field extension, and let \( A \) and \( B \) be \( n \)-dimensional \( K \)-subspaces of \( L \) with \( n > 0 \).  
An ordered basis \( \mathcal{A} = \{a_1, \ldots, a_n\} \) of \( A \) is said to be \emph{matched} to an ordered basis \( \mathcal{B} = \{b_1, \ldots, b_n\} \) of \( B \) if, for each \( i \in [n] \),
\[
a_i^{-1}A \cap B \,\subseteq\, \langle b_1, \ldots, b_{i-1}, b_{i+1}, \ldots, b_n \rangle.
\]
Note that when the above condition holds, we have \( a_i b_i \notin \mathcal{A} \) for all \( i \), so the correspondence \( a_i \mapsto b_i \) defines a matching, in the group-theoretic sense, from \( \mathcal{A} \) to \( \mathcal{B} \) in the multiplicative group of \( L \).

We say that the subspace \( A \) is \emph{matched} to \( B \) if every ordered basis of \( A \) can be matched to some ordered basis of \( B \).  
In this case, we also say that the pair \( (A,B) \) is \emph{matchable}; otherwise, we call \( (A,B) \) \emph{unmatchable}.

A necessary condition for \( A \) to be matched to \( B \) is \( 1 \notin B \). This is discussed in detail in \cite{Eliahou 2}.

A field extension \( K \subsetneq L \) is said to have the \emph{linear matching property} if, for every pair of finite-dimensional \( K \)-subspaces \( A \) and \( B \) of \( L \) with \( \dim A = \dim B > 0 \) and \( 1 \notin B \), the subspace \( A \) is matched to \( B \).

Eliahou and Lecouvey established in \cite{Eliahou 2} the following fundamental results:
\begin{itemize}
    \item A subspace \( A \) is matched to itself if and only if \( 1 \notin A \).
    \item A field extension \( K \subsetneq L \) has the linear matching property if and only if \( L \) contains no nontrivial proper finite-dimensional extension over \( K \).
\end{itemize}

In Section~\ref{Linear structure section}, we will present a decomposition theorem for unmatchable pairs of subspaces (Theorem~\ref{linearStruc}). Several consequences of this result will then be developed.  We will also give a brief summary of the correspondence between relevant concepts in the group and linear settings (see Table~\ref{tab:decomp-dictionary-grid}).

\section{Structure theorem for unmatchable pairs in abelian groups} \label{Structure section groups}

The theorem below is one of the main results of \cite{Aliabadi 5}.  It can be used to recover, efficiently and without recourse to supplementary ideas, all previously known results about matchings in abelian groups.  In this section, it will be used to derive a structure theorem for unmatchable pairs \( (A, B) \).

\begin{theorem}  \label{Complete matchings for groups}
Let \( A \) and \( B \) be nonempty finite subsets of an abelian group \( G \), with \( |A| = |B| \) and \( 1 \notin B \).  Then \( (A,B) \) is matchable if and only if, for every pair of nonempty subsets \( S \subseteq A \) and \( R \subseteq B \cup \{ 1 \} \) with \( SR = S \), we have 
\[ |S| \leq |B \setminus R|. 
\]
\end{theorem}

For any subset \( X \) of an abelian group \( G \), we will write \( \langle X \rangle \) for the subgroup of \( G \) generated by \( X \).

The condition \( SR = S \) in Theorem~\ref{Complete matchings for groups} can be formulated in a different way, as described in the following proposition. For a proof, see \cite{Aliabadi 5}.

\begin{proposition} \label{UnionOfCosets}
Let \( G \) be an abelian group and let \( S \) and \( R \) be nonempty finite subsets of \( G \). Then \( SR = S \) if and only if \( S \) can be written as a disjoint union of cosets of \( \langle R \rangle \). 
\end{proposition}

Below, we state and prove our structure theorem for unmatchable pairs. This result offers a different and complementary perspective to that of Theorem~\ref{Complete matchings for groups}.  It shows that unmatchability can be understood not just as the failure of a Hall-type condition, but as the consequence of a structural feature of the pair \((A,B)\).  More precisely, unmatchability occurs when there exists a nonempty set \(R\subseteq B\) such that $A$ is nearly a disjoint union of cosets of the subgroup $\langle R\rangle$, with a remainder set $Y$ of smaller size than  $R$.  This viewpoint yields concrete certificates of unmatchability, shows that several seemingly disparate matching results actually have a common underlying source, and serves as the main engine for new consequences in this paper (including existence criteria for nontrivial unmatchable pairs).

\begin{theorem} \label{Structure groups}
Let \( A \) and \( B \) be nonempty finite subsets of an abelian group \( G \), with \( |A| = |B| \).  Then the pair \( (A , B) \) is unmatchable if and only if there exists a nonempty subset \( R \) of \( B \) such that \( A \) and \( B \) can be expressed as disjoint unions as follows:
\[
A = S \cup Y, \qquad 
B = R \cup Z,
\]
where \( S \) is a disjoint union of cosets of \( \langle R \rangle \), and \( Y \) satisfies \( |Y| < | R | \). 
\end{theorem}

\noindent \emph{Note: }The conditions on the sets in the above decomposition ensure that \( S \neq \emptyset \).

\begin{proof}
Assume that the pair \( (A , B) \) is unmatchable.  Observe that if \( 1 \in B \), then we can simply take \( S = A \), \( Y = \emptyset \), \( R = \{ 1 \} \), and \( Z = B \setminus R \).  

Suppose \( 1 \notin B \). Then, by Theorem~\ref{Complete matchings for groups}, there exist sets \( S \) and \( R \) satisfying the following conditions: 
\begin{itemize}
\item[(a)] \( \emptyset \neq  S \subseteq A \) and  \( \emptyset \neq R \subseteq B \cup \{ 1 \} \),
\item[(b)] \( SR = S \), 
\item[(c)] \( |S| > |B \setminus R | \).
\end{itemize}
We see from (b) and Proposition~\ref{UnionOfCosets} that \( S \) is a disjoint union of cosets of \( \langle R \rangle \).  Note that \( R \neq \{ 1 \} \), by (c). Moreover, we may assume that  \( 1 \notin R \), since otherwise we can replace \( R \) with \( R \setminus \{ 1 \} \), and (a)--(c) will still hold.

Thus \( R \subseteq B \), and we can now rewrite (c) as
\[
|R| > |B| - |S|.
\]
Take \( Y = A \setminus S \) and \( Z = B \setminus R \). Since \( |A| = |B| \) and \( S \subseteq A \), the above inequality shows that \( |Y| < |R| \), giving us the desired decomposition.

Conversely, assume that \( A \) and \( B \) can be written as in the statement. If \( B \) contains \( 1 \), then clearly \( (A,B) \) is unmatchable, so assume \( 1 \notin B \). The sets \( S \) and \( R \) satisfy \( \emptyset \neq S \subseteq A \), \( \emptyset \neq R \subseteq B \), and \( SR = S \).  Furthermore, the inequality \( |Y| < | R | \) implies \( |S| + |R| > |S| + |Y| = |A| = |B| \), hence \( |S| > |B \setminus R| \).  Therefore, by Theorem~\ref{Complete matchings for groups}, the pair \( (A,B) \) is unmatchable.
\end{proof}

Let \(G\) be an abelian group and let \( H \) be a subgroup. A nonempty subset \(X \) of \( G \) is called $H$-\emph{periodic} if \(X\) can be written as a disjoint union of cosets of \(H\). Any decomposition as in Theorem~\ref{Structure groups},
\[
A=S\cup Y,\qquad B=R\cup Z,
\]
will be called a \emph{nearly periodic decomposition} for the unmatchable pair \((A,B)\), with $\langle R\rangle$-periodic set \(S\) and \emph{remainder} $Y.$ 

We remark that the set \( Z \) in a nearly periodic decomposition will play a passive role in this paper, receiving attention only when we are constructing unmatchable pairs \( (A,B) \) with \( 1 \notin B \), as we do in the proof of Theorem~\ref{BoundaryGroup}.

\begin{corollary}\label{Generalize Symmetric}
Let \( A \) and \( B \) be nonempty finite subsets of an abelian group $G$, with \( |A| = |B| \) and \( 1 \notin A \). Assume that for each nonempty subset \( R \) of \( B \), we have \( | \langle R \rangle \cap A | \geq |R| \).  Then the pair \( (A,B) \) is matchable.
\end{corollary}

\begin{proof}
Assume that $(A,B)$ is unmatchable.  Then the pair \( (A, B) \) has a nearly periodic decomposition
\[
A = S \cup Y, 
\qquad 
B = R \cup Z,
\]
with \( R, S, Y, Z \) as described in Theorem~\ref{Structure groups}. Note that \( \langle R \rangle \cap S = \emptyset \), since \( 1 \notin A \) and \( S \) is a union of cosets of \( \langle R \rangle \). Therefore 
\[ 
| \langle R \rangle \cap A | = | \langle R \rangle \cap Y | \leq |Y| < | R |.
\]
\end{proof}

\begin{remark}  \label{classical symmetric}
If we take \( B = A \) in Corollary~\ref{Generalize Symmetric}, with \( 1 \notin A \), we find that the hypothesis is satisfied and therefore the pair \( (A,A) \) is matchable. In this way, we recover Theorem~2.1 from \cite{Losonczy 2}. 
\end{remark}

A nonempty finite subset \( B \) of a group \( G \) is called a {\em Chowla set} if \( o(x) > |B| \) for every \( x \in B \), where \( o(x) \) denotes the order of \( x \). In \cite{Hamidoune}, Hamidoune showed that if \( B \) is a Chowla set, then the pair \( (A,B) \) is matchable for any set \( A \) with \( |A| = |B| \).   Using the structure theorem above, it is possible to obtain this fact quickly when \( G \) is abelian.

\begin{corollary}  \label{Chowla}
Let \( A \) and \( B \) be nonempty finite subsets of an abelian group \( G \), with \( |A| = |B| \).  Assume that \( B \) is a Chowla set.  Then the pair \( (A,B) \) is matchable.
\end{corollary}

\begin{proof}
By the Chowla assumption, for each nonempty subset \( R \) of \( B \), we have \( | \langle R \rangle | > |B| = |A| \), and so it is impossible for \( A \) to be decomposed as in the statement of Theorem~\ref{Structure groups}.
\end{proof}

\begin{definition}  
For any group \( G \) having at least one subgroup \( H \) such that \( \{ 1 \} \subsetneq H \subsetneq G \) and \( |H| < \infty \), we define \( n_0(G) \) by 
\[
n_0(G) = \min_H \, |H|,
\]
where \( H \) ranges over all subgroups of \( G \) as described above.
\end{definition}

The result below establishes a necessary and sufficient condition for the existence of nontrivial unmatchable sets \( A, B \) whenever \( n_0(G) \) is defined and less than \( |G| \).

\begin{theorem}\label{BoundaryGroup}
Let \( G \) be an abelian group with at least one finite nontrivial proper subgroup.  Let \( n \) be an integer such that \( n_0(G) \leq n < |G| \). Then there exist subsets \( A, B \subseteq G \) such that \( 1 \notin B \), \( |A| = |B| = n \), and \( (A,B) \) is an unmatchable pair if and only if there is a subgroup \( H \) of  \( G \) with \( |H|\leq n \) and \( |H| \nmid (n+1) \).
\end{theorem}

\begin{proof}
Suppose an unmatchable pair $(A,B)$ as described in the statement exists. Then, by Theorem~\ref{Structure groups}, there is a nonempty subset $R$ of $B$ such that $A$ and $B$ can be expressed as disjoint unions
\[
A = S \cup Y, 
\qquad 
B = R \cup Z,
\]
where $S$ is a nonempty disjoint union of cosets of $\langle R \rangle$, and the remainder $Y$ satisfies $|Y| < |R|$.  We will show that $\langle R \rangle$ can serve as the desired subgroup \( H \) of $G$. 

First note that \( \langle R \rangle \) is finite, since \( S \) is.  Let \( m = |\langle R \rangle| \) and write \( |S| = mq \), where $q \geq 1$. Then
\[ n = |S| + |Y| = mq + |Y|.
\]
Clearly \( m \leq mq \leq n \). We also have \( |Y| < |R| \leq m-1 \), as \( 1\notin R \), and so 
\[
mq \leq n \leq mq + m-2.
\]
Hence
\[
mq < mq + 1 \leq n + 1 \leq mq + m-1 < m(q+1).
\]
From this we see that $m \nmid (n+1)$, as desired.

Conversely, assume that $G$ has a subgroup $H$ with $|H| \leq n$ and $|H| \nmid (n+1)$. Let \( m = |H| \) and note that \( m \geq n_0(G) > 1 \). We can write 
\[
n = mq + r 
\]
for some integers $q\geq 1$ and $0\leq r \leq m-2$. Set $R=H\setminus\{1\}$ and let $S$ be a disjoint union of $q$ distinct cosets of $\langle R \rangle$. Then $|S| = mq$. Let $Y \subseteq G \setminus S$ with $|Y| = r$, and let $Z \subseteq G \setminus (R \cup \{1\})$ with $|Z|=n-m+1$. Finally, define
\[
A = S \cup Y, 
\qquad 
B = R \cup Z.
\]
Then $1 \notin B$ and $|A|=|B|=n$.  Further, all of the conditions in Theorem~\ref{Structure groups} are satisfied, implying that $(A,B)$ is unmatchable.
\end{proof}

\begin{example}
Using Theorem~\ref{BoundaryGroup}, we can see that in the additive group \( \mathbb{Z}/4\mathbb{Z} \), every pair \( (A,B) \) of $3$-element subsets with \( \bar{0} \notin B \) is matchable, since \( \mathbb{Z}/4\mathbb{Z} \) has no subgroup \( H \) satisfying \( |H| \leq 3 \) and \( |H| \nmid 4 \).  On the other hand, in \( \mathbb{Z}/12\mathbb{Z} \), the $5$-element subsets
\[
A=\{\bar{0},\bar{1},\bar{3},\bar{6},\bar{9}\}, 
\qquad 
B=\{\bar{1},\bar{2},\bar{3},\bar{6},\bar{9}\}
\]
(with $\bar{0} \notin B$) form an unmatchable pair.  This is consistent with Theorem~\ref{BoundaryGroup}, since the group \( \mathbb{Z}/12\mathbb{Z} \) does have a subgroup \( H \) satisfying \( |H| \leq 5 \) and \( |H| \nmid 6 \), namely \( H = \{ 0,3,6,9 \} \).
\end{example}

In the following corollary, we establish a connection between unmatchable sets in a given abelian group  \( G \) and their images in a quotient of \( G \).

\begin{corollary}\label{quotient}
Let $G$ be an abelian group and let $(A,B)$ be a pair of nonempty finite subsets with $|A|=|B|$. Suppose $(A,B)$ is unmatchable, as witnessed by the decomposition given in Theorem~\ref{Structure groups}:
\[
A = S \cup Y,
\]
where $S$ is a disjoint union of cosets of the subgroup $\langle R\rangle$ for some nonempty $R\subseteq B$, and $|Y|<|R|$.  Suppose $H$ is a subgroup of $G$ satisfying $H\cap (SS^{-1}\cup RR^{-1})=\{1\}$, and let $\pi : G\to G/H$ be the corresponding quotient map. Then the projected pair $(\pi(A),\pi(B))$ is unmatchable in $G/H$.
\end{corollary}

\begin{proof}
First notice that if $|\pi(A)|\neq |\pi(B)|$, then the pair $(\pi(A),\pi(B))$ is unmatchable and we are done. Suppose $|\pi(A)|= |\pi(B)|$.  Define $Y'=\pi(A)\setminus \pi(S)$ and $Z'=\pi(B)\setminus \pi(R)$. Then we obtain the decompositions
\[
\pi(A)=\pi(S)\cup Y' \quad\text{and}\quad \pi(B)=\pi(R)\cup Z'.
\]
These satisfy the conditions of Theorem \ref{Structure groups}, since 
\begin{itemize}
    \item $\pi(S)$ is a nonempty disjoint union of cosets of $\langle \pi(R)\rangle$,
    \item $|Y'|<|\pi(R)|$.
\end{itemize}
To see why the latter holds, observe first that 
\begin{align*}
|Y'|=|\pi(A)|-|\pi(S)|\leq |A|-|\pi(S)|. 
\end{align*}
Now, \( |\pi(S)|=|S| \) and \( |\pi(R)| = |R| \) since \( H\cap (SS^{-1}\cup RR^{-1})=\{1\} \). Combining this with the above, we obtain
\[
|Y'|\leq |A|-|S|=|Y|<|R|=|\pi(R)|.
\]
Therefore, $(\pi(A), \pi(B))$ is unmatchable in $G/H$.
\end{proof}

\begin{remark}
The assumption $H\cap (SS^{-1}\cup RR^{-1})=\{1\}$ in Corollary~\ref{quotient} cannot be dropped. To see this, consider the additive group $G=\mathbb{Z}/8\mathbb{Z}$, the subgroup \( H = \{\bar{0},\bar{4}\} = 4\mathbb{Z}/8\mathbb{Z} \), and the corresponding quotient map 
\[
\pi:G\rightarrow G/H\;\cong\;\mathbb{Z}/4\mathbb{Z}.
\]
Define
\[
A=\{\bar{0},\bar{1},\bar{2},\bar{4},\bar{6}\},\qquad 
B=\{\bar{1},\bar{2},\bar{3},\bar{5},\bar{6}\}.
\]
The following decomposition of $A$ and $B$ satisfies the assumptions of Theorem~\ref{Structure groups}: 
\[
A=S\cup Y,\qquad B=R\cup Z,
\]
where $R=\{\bar{2},\bar{6}\}$, $S=\bar{0}+\langle R\rangle=\{\bar{0},\bar{2},\bar{4},\bar{6}\}$, 
$Y=\{\bar{1}\}$, and $Z=\{\bar{1},\bar{3},\bar{5}\}$. 
Hence $(A,B)$ is unmatchable in $\mathbb{Z}/8\mathbb{Z}$. Note that  $H\cap \big( (S-S)\cup (R-R)\big) \neq\{\bar{0}\}$.

We now compute $\pi(A)$ and $\pi(B)$:
\[
\pi(A)=\{\bar{0},\bar{1},\bar{2}\}\subseteq \mathbb{Z}/4\mathbb{Z},\qquad 
\pi(B)=\{\bar{1},\bar{2},\bar{3}\}\subseteq \mathbb{Z}/4\mathbb{Z}.
\]
These are matchable: for instance, the bijection $\bar{0}\mapsto \bar{3}$, $\bar{1}\mapsto \bar{2}$, $\bar{2}\mapsto \bar{1}$ is a matching since each sum lands at $\bar{3}$, which lies outside $\pi(A)$. 

Thus $(A,B)$ is unmatchable in $G$, but $(\pi(A),\pi(B))$ is matchable in $G/H$. 
\end{remark}

\section{Structure theorem for unmatchable subspaces in a field extension}\label{Linear structure section}

Let \( K \subseteq L \) be a field extension.  Given a $K$-subspace \( S \) of \( L \), we will write \( \dim_K S \) or \( \dim S \) for its $K$-dimension, preferring the latter in situations where no confusion can arise. 

The following theorem, a main result of \cite{Aliabadi 5}, provides a framework for studying matchable subspaces in field extensions.  We will apply it to obtain a structure theorem for unmatchable pairs $(A,B)$.

\begin{theorem}\label{main linear}
Let \( K \subsetneq L \) be a field extension, and let \( A \) and \( B \) be two \( n \)-dimensional \( K \)-subspaces of \( L \), with \( n > 0 \) and \( 1 \notin B \). Then \( A \) is matched to \( B \) if and only if, for every pair of nonzero \( K \)-subspaces \( S \subseteq A \) and \( R \subseteq B\oplus K \) with \( \langle SR \rangle = S \), we have
\[
\dim S\leq \dim \bigl(B/(R \cap B)\bigr).
\]
\end{theorem}

The proposition below gives an indication of how the condition \( \langle SR \rangle = S \) in Theorem~\ref{main linear} will be used in the rest of this section.  In the proposition and elsewhere, we make use of the following notation and terminology from field theory.  Let \( K \subseteq L \) be a field extension and let \( R \) be a subset of \( L \).  We write \( K(R) \) for the smallest subfield of \( L \) containing the set \( K \cup R \).  If \( R = \{ x_1, \ldots , x_n \} \), we may also write \( K(x_1,\ldots ,x_n)\) for \( K(R) \), and we write \( K[x_1,\ldots ,x_n] \) for the smallest subring of \( L \) containing \( K \cup \{x_1, \ldots,x_n \} \).

For any intermediate field \( F \) of the extension \( K \subseteq L \), we write \( [F:K] \) for the $K$-dimension of \( F \). We call this dimension the \emph{degree} of \( F \) over \( K \).  For \( x \in L \), we say that \( x \) is \emph{algebraic} over \( K \) if the degree of \( K(x) \) over \( K \) is finite.

\begin{proposition} \label{proposition linear}
Let \( K \subseteq L \) be a field extension, let \( n \) be a positive integer, and let \( S \) and \( R \) be nonzero \( K \)-subspaces of \( L \).  Assume that \( \langle SR \rangle = S \) and  \( \dim_K S \leq n \). The following statements hold: 
\begin{itemize}
\item[(i)] Let \( x \in R \). Then \( aK(x) \subseteq S \) for all \( a \in S \). In addition, \( [K(x):K] \leq n\), so that \( x \) is algebraic over $K$.    
\item[(ii)] We in fact have \( aK(R) \subseteq S \) for all \( a \in S \). Also, \( [K(R):K] \leq n \) and \( \dim_K S \) is a positive multiple of \( [K(R):K] \).
\end{itemize}
\end{proposition}

\begin{proof}
Part (i) is proved in \cite{Aliabadi 5}. For (ii), observe that for any \( a \in S\setminus \{ 0 \} \), we have 
\( \dim_K R = \dim_K aR \) and \( aR \subseteq \langle SR \rangle = S \).  Hence \( \dim_K R \leq \dim_K S \leq n \). Let \( \{ x_1, \ldots , x_t \} \) be a $K$-basis for \( R \). According to (i), each basis element \( x_j \) is algebraic over \( K \), giving us \( K(R) = K(x_1, \ldots ,x_t) = K[x_1, \ldots ,x_t] \).  We also have \( aK(x_j) \subseteq S \) for all \( a \in S \) and $j \in [t]$, again by (i). It is now clear that \( aK(R) \subseteq S \) for all \( a \in S \). In other words, \( S \) can be regarded as a vector space over \( K(R) \).  

For the rest of (ii), consider the equation 
\[
\dim_K S = \dim_{K(R)} S \cdot [K(R):K]. 
\]
Since \( 0 < \dim_K S \leq n < \infty\), it follows that \( [K(R):K] \leq n \) and \( \dim_K S \) is a positive multiple of the degree \( [K(R):K] \).
\end{proof}

We are ready to state and prove one of the main results of this section, giving a structural decomposition for unmatchable pairs \((A,B)\) over a field extension \( K \subsetneq L \). This theorem is a linear analogue of Theorem~\ref{Structure groups}, but there is more to notice about it. In the field extension setting, matchability is controlled by the lattice of intermediate fields: unmatchability forces $A$ to be almost a sum of translates of an intermediate field $K(R)$ for some $K$-subspace $R$, with a remainder subspace $Y$ having smaller dimension than $R$.  Equivalently, if we view $L$ as a vector space over $K(R)$ via multiplication, Theorem~\ref{linearStruc} shows that unmatchability forces a substantial portion of $A$ to be a $K(R)$-subspace of $L$, again with a small remainder subspace $Y$.  This yields a structural characterization of the failure of matchability which depends on the extension $K\subsetneq L$ and leads to results that are most naturally formulated in terms of properties of intermediate extensions (for instance, Corollary~\ref{AnalogueGenSymmetric} and Theorem~\ref{BoundaryLin}).

\begin{theorem}\label{linearStruc}
Let \( K \subsetneq L \) be a field extension, and let \( A \) and \( B \) be two \( n \)-dimensional \( K \)-subspaces of \( L \), with \( n > 0 \). Then the pair \( (A, B) \) is unmatchable if and only if there exists a nonzero $K$-subspace \( R \) of \( B \) such that \( A \) and \( B \) can be written as direct sums
\[
A = S \oplus Y, \qquad
B = R \oplus Z,
\]
where \( \dim Y < \dim R \) and
\[
S = \bigoplus_{i=1}^m a_i K(R)
\]
for some \( a_1, \ldots ,a_m \in S \).  Here, \( m = \dim_{K(R)}S = (\dim_K S) / [K(R):K] \).
\end{theorem}

\noindent \emph{Note: }The dimension bound on $Y$ guarantees that \( S \neq \{ 0 \}\).

\begin{proof}
Assume that the pair \( (A,B) \) is unmatchable. If \(1\in B\), define \(S = A\), \(Y = \{0\}\), \(R = K\), and choose \(Z\) to be a \( K \)-linear complement of \(K\) in \(B\) (so that \(B = K\oplus Z\)). The required conditions are then satisfied (we can take \( a_1,\ldots, a_m  \) in the statement to be the elements of any $K$-basis of \( A \)).

Now suppose \( 1 \notin B \). By Theorem \ref{main linear}, there exist nonzero subspaces \( S \) and \( R \) satisfying the following:
\begin{itemize}
    \item[(a)] \( S \subseteq A \) and \( R \subseteq B \oplus K \),
    \item[(b)] \( \langle SR \rangle = S \),
    \item[(c)] \( \dim S > \dim \bigl( B/(R \cap B)\bigr) \).
\end{itemize}
By condition (b) and Proposition~\ref{proposition linear}, we have \( aK(R) \subseteq S \) for all \( a \in S \). Thus \( S \) can be regarded as a vector space over \( K(R) \). As such, it has finite dimension, since \( \dim_{K(R)}S \cdot [K(R):K] = \dim_K S \leq n\). Hence there exist \( a_1, \ldots , a_m \in S \) such that 
\[ 
S = \bigoplus_{i=1}^m a_i K(R), 
\]
where \( m = \dim_{K(R)}S = (\dim_K S) / [K(R):K] \).  Note that \( R \neq K \) by (c), and we may assume that \( R \subseteq B \); otherwise, we can replace \( R \) by \( \pi_B(R) \), where \( \pi_B \) is the projection map \( B \oplus K \rightarrow B \) along \( K \), and conditions (a)–(c) will still hold. Hence condition (c) gives us
\[
\dim R > \dim B - \dim S.
\]
Now let \( Y \) be a complementary summand of \( A \) relative to \( S \), and \( Z \) be a complementary summand of \( B \) relative to \( R \). Since \( \dim A = \dim B = n \), it follows from the inequality above that \( \dim Y < \dim R \), as desired.

Conversely, assume that \( A \) and \( B \) can be written as in the statement. If \( 1 \in B \), then the pair \( (A, B) \) is unmatchable, so suppose \( 1 \notin B \). Then \( S \) and \( R \) are nonzero subspaces with \( S \subseteq A \), \( R \subseteq B \), and \( \langle SR \rangle = S \). Furthermore, the inequality \( \dim Y < \dim R \) implies
\[
\dim R + \dim S > \dim Y + \dim S = \dim A = \dim B.
\]
This yields
\[
\dim S > \dim B - \dim R = \dim (B/R)  = \dim \bigl(B/(B\cap R)\bigr).
\]
Therefore, by Theorem \ref{main linear}, the pair \( (A, B) \) is unmatchable.
\end{proof}

In Table~\ref{tab:decomp-dictionary-grid} below, we summarize the main information regarding unmatchable pairs in abelian groups and over field extensions.

\begin{table}[ht] 
\small
\centering
\renewcommand{\arraystretch}{1.15}
\begin{tabularx}{\linewidth}{|l|>{\raggedright\arraybackslash}X|>{\raggedright\arraybackslash}X|}
\hline
& \textbf{Group setting} & \textbf{Linear setting} \\
\hline
Objects
& $A,B\subseteq G$, $0<|A|=|B|<\infty$
& \mbox{$K$-subspaces $A,B \subseteq L$}, $0 < \dim A=\dim B < \infty$ \\
\hline
Decomposition
& $A=S\cup Y$, $B=R\cup Z$ (disjoint unions)
& $A=S\oplus Y$, $B=R\oplus Z$ (direct sums) \\
\hline
Structure of $S$
& $S=a_1\langle R\rangle \cup \cdots \cup a_m \langle R \rangle $ (pairwise disjoint union)
& $S=a_1K(R)\oplus \cdots \oplus a_mK(R)$  (direct sum)\\
\hline
Gap condition
& $|Y|<|R|$
& $\dim Y<\dim R$ \\
\hline
Approach
& Find $\emptyset \neq R\subseteq B$ with $S$ as above and $|Y|<|R|$
& Find $\{0\}\neq R\subseteq B$ with $S$ as above and $\dim Y<\dim R$ \\
\hline
Consequence
& $(A,B)$ is an unmatchable pair
& $(A,B)$ is an unmatchable pair \\
\hline
\end{tabularx}
\vspace{0.4cm}
\caption{Side-by-side dictionary for the decomposition of unmatchable pairs.}
\label{tab:decomp-dictionary-grid}
\end{table}

For our first application of Theorem~\ref{linearStruc}, we give an alternate proof of Theorem~4.15 from \cite{Aliabadi 4}.

\begin{corollary}
Let \( K \subsetneq L \) be a field extension, and let \( A \) and \( B \) be two \( n \)-dimensional \( K \)-subspaces of \( L \), with \( n > 0 \).  
Suppose that \( A \) is matched to \( B \). Then \( \langle AB \rangle \neq A \).
\end{corollary}

\begin{proof}
Assume that \( \langle AB \rangle = A \).  Set \( S = A \), \( R = B \), and \( Y = Z = \{0\} \). By Proposition~\ref{proposition linear}, we have \( aK(R) \subseteq S \) for all \( a \in S \).  It is now easy to see that the sets \( S, R, Y, Z \) satisfy all of the hypotheses of Theorem~\ref{linearStruc}.  Therefore \( A \) is not matched to \( B \).
\end{proof}

Let \( K \subsetneq L \) be a field extension, and let \( B \) be a \( K \)-subspace of \( L \).  
We say that \( B \) is a \emph{Chowla subspace} if, for every \( x \in B \setminus \{0\} \),  
\[
[K(x) : K] \geq \dim B + 1.
\]

In particular, if \( B \) is a Chowla subspace, then \( 1 \notin B \). The next result, originally proved in \cite{Aliabadi 5}, is retrieved easily using Theorem~\ref{linearStruc}.

\begin{corollary} \label{Chowla subspace}
Let \( K \subsetneq L \) be a field extension, and let \( A \) and \( B \) be two finite-dimensional \( K \)-subspaces of \( L \), with \( \dim A = \dim B > 0 \). Assume that \( B \) is a Chowla subspace. Then \( A \) is matched to \( B \).
\end{corollary}

\begin{proof}
For any \( a \in A\setminus \{ 0 \} \), any nonzero $K$-subspace \( R \subseteq B \), and any \( x \in R\setminus \{ 0 \} \), we have 
\[ 
\dim aK(R) \geq \dim aK(x) = [K(x):K] > \dim B = \dim A,
\]
by the Chowla assumption. Thus a decomposition of \( A \) as in Theorem~\ref{linearStruc} is not possible. We conclude that \( A \) is matched to \( B \).
\end{proof}

\begin{example}
Take \( K=\mathbb{Q} \) and \( L=\mathbb{Q}(\sqrt[p]{2}) \), where \( p \) is a prime.  Let $A$ be any $\mathbb{Q}$-subspace of \( L \) of dimension \( p-1 \), and let \( B \) be the trace-zero $\mathbb{Q}$-subspace of \( \mathbb{Q}(\sqrt[p]{2}) \), which has dimension $p-1$. 

We claim that \( B \) is a Chowla subspace. To see this, observe that the trace in this extension of any \( a \in \mathbb{Q} \) is \( pa \) (see, for example, Chapter 1 of \cite{McCarthy}).  Hence if \( x \in B \setminus \{ 0 \} \), then we must have \( x \in \mathbb{Q}(\sqrt[p]{2}) \setminus \mathbb{Q} \). It follows that for any \( x \in B\setminus \{ 0 \} \), the degree \( [\mathbb{Q}(x):\mathbb{Q}] \) cannot equal \( 1 \) and must divide \( [\mathbb{Q}(\sqrt[p]{2}):\mathbb{Q}] = p \).  Therefore \( [\mathbb{Q}(x):\mathbb{Q}] = p \), which is greater than \( \dim B \).  The claim is established.  

We conclude that, by Corollary~\ref{Chowla subspace}, $A$ is matched to $B$.
\end{example}

The following result is new.  It can be viewed as a linear analogue of Corollary~\ref{Generalize Symmetric}.

\begin{corollary}  \label{AnalogueGenSymmetric}
Let \( K \subsetneq L \) be a field extension, and let \( A \) and \( B \) be two \( n \)-dimensional \( K \)-subspaces of \( L \), with \( n > 0 \) and \( 1 \notin A \). Assume that for every nonzero $K$-subspace \( R \) of \( B \), we have 
\[
\dim (K(R) \cap A) \geq \dim R.
\]
Then the pair \( (A, B) \) is matchable.
\end{corollary}

\begin{proof}
Suppose that \( (A, B) \) is unmatchable. Then, by Theorem~\ref{linearStruc}, there exists a decomposition
\[
A = S \oplus Y, 
\qquad 
B = R \oplus Z,
\]
where \( R \) is a nonzero $K$-subspace, \( S \) is stable under multiplication by elements of \( K(R)\), and \( \dim Y < \dim R \).  Observe that  \( K(R) \cap S = \{0\}\); otherwise, we can choose a nonzero \( a \in K(R) \cap S \), and then \( 1 = aa^{-1} \in aK(R) \subseteq S \subseteq A \), contradicting \( 1 \notin A \). Consider now the projection map \( \pi : A \to Y \) along \(S\), i.e., the $K$-linear map satisfying 
\( \ker \pi = S \) and  \( \pi|_{Y} = \mathrm{id}_{Y}\). Restrict \(\pi\) to the subspace \( T = K(R) \cap A \). Then
\[
\ker(\pi|_{T})
    = T \cap \ker \pi
    = (K(R) \cap A) \cap S
    = K(R) \cap S
    = \{0\}.
\]
Thus \(\pi|_{T}\) is injective, and we obtain
\[
\dim (K(R) \cap A) = \dim T = \dim \pi(T) \leq \dim Y < \dim R.
\] 
\end{proof}

\begin{remark}
Taking \( B = A \) in Corollary~\ref{AnalogueGenSymmetric}, with \( 1 \notin A \), we find that \( A \) must be matched to itself.  On the other hand, it follows immediately from the definition of matchable pair that \( 1 \notin A \) is a necessary condition for \( (A,A) \) to be matchable.  We thus obtain the commutative case of Theorem 2.8 in \cite{Eliahou 2}.  
\end{remark}

\begin{definition}
Let \( K \subsetneq L \) be a field extension with at least one nontrivial intermediate field of finite $K$-dimension.  We define \( n_0(K, L) \) to be the smallest degree of such an extension, i.e.,
\[
n_0(K, L) = \min_F \, [F : K],
\]
where \( F \) ranges over all intermediate fields \( K \subsetneq F \subseteq L \) with \( [F : K] < \infty \).
\end{definition}

In Theorem~\ref{BoundaryLin} below, we provide a necessary and sufficient condition for the existence of a nontrivial unmatchable pair \( (A,B) \) in all situations where \( n_0(K,L) \) is defined and less than \( [L:K] \). This result is a linear analogue of Theorem~\ref{BoundaryGroup}.

\begin{theorem}\label{BoundaryLin}
Let $K \subsetneq L$ be a field extension with at least one proper nontrivial intermediate field of finite $K$-dimension. Let $n$ be a positive integer such that 
\[
n_0(K,L) \leq n < [L:K].
\]
Then there exist $n$-dimensional $K$-subspaces $A,B \subseteq L$ with $1 \notin B$ such that $(A,B)$ is an unmatchable pair if and only if there 
exists an intermediate extension $K \subsetneq F \subseteq L$ with 
$[F:K] \leq n$ and $[F:K] \nmid (n+1)$.
\end{theorem}

\begin{proof}
Assume there exist $n$-dimensional $K$-subspaces $A,B \subseteq L$ such that $1 \notin B$ and \( (A,B) \) is unmatchable. Applying Theorem~\ref{linearStruc}, we obtain subspaces $R,S,Y,Z$ as described therein, with \( S \) a nonzero direct sum of $K$-subspaces of the form \( aK(R) \) (where \( a \in S \)).  Thus \( [K(R):K] \leq \dim S \leq n \). 

Let \( F = K(R) \) and \( m = [F:K] \). We now show that $F$ can serve as the desired intermediate field of $K \subsetneq L$.  The $K$-dimension of \( S \) must be a positive multiple of \( m \), on account of the direct sum structure described above.   Write \( \dim S = mq \) where \( q \geq 1 \). Then 
\[
n = \dim A = \dim S + \dim Y = mq + \dim Y.
\]
Clearly \( m \leq mq \leq n \). Observe also that since $\dim Y < \dim R$, 
$R \subseteq K(R) = F$, and $R \cap K = \{0\}$, we have $\dim Y \leq m-2$. This gives us 
\[
mq \leq n \leq mq + m-2.
\]
Hence
\[
mq < mq + 1 \leq n+1 \leq mq + m-1 < m(q+1),
\]
so that $m \nmid (n+1)$, as desired.

Conversely, assume that there exists an intermediate extension $K \subsetneq F \subseteq L$ with $[F:K] \leq n$ and $[F:K] \nmid (n+1)$. Let $m = [F:K]\geq 2$ and observe that we can write 
\[
n = mq + r,
\]
where $q \geq 1$ and $0 \leq r \leq m-2$. 

Next we define \( R \) and \( S \). Choose a $K$-subspace $R \subseteq F$ with $\dim R = m-1$ and $R \cap K = \{0\}$. 
Note that \( [L:F] > n/m \geq q \), and let $a_1, \dots, a_q \in L$ be linearly independent over \( F \). Define  
\[
S = \sum_{i=1}^q a_i F.
\] 
This sum is direct. Let $Y$ be a $K$-subspace of $L$ with $Y \cap S = \{0\}$ and $\dim Y = r$. 
Also, let $Z$ be a $K$-subspace of $L$ with $Z \cap (R+K) = \{0\}$ and \( \dim Z = n-m+1 \). 

Finally, define 
\[
A = S + Y, \qquad B = R + Z.
\]
Both of these sums are direct, and \( 1 \notin B \). Moreover, the conditions of Theorem~\ref{linearStruc} are all satisfied, 
so that the pair $(A,B)$ is unmatchable.
\end{proof}

\begin{example}
Using Theorem~\ref{BoundaryLin}, we can show that in the field extension \(\mathbb{Q} \subseteq \mathbb{Q}(\sqrt[4]{2})\), every pair \((A,B)\) of $3$-dimensional $\mathbb{Q}$-subspaces of \(\mathbb{Q}(\sqrt[4]{2})\) with \(1 \notin B\) is matchable. Note that \( [\mathbb{Q}(\sqrt[4]{2}): \mathbb{Q}] = 4\) and this extension has exactly one nontrivial proper intermediate field, namely \(\mathbb{Q}(\sqrt{2})\).  Hence \( n_0(\mathbb{Q},\mathbb{Q}(\sqrt[4]{2}))\) is defined and equals \( [\mathbb{Q}(\sqrt{2}): \mathbb{Q}] = 2 \).  Taking \( n = 3 \) in Theorem~\ref{BoundaryLin} and noting that \( 2 \mid (n+1) \), we conclude that every pair of subspaces \( (A,B) \) as above is matchable.
\end{example}

The corollary below addresses the special case of finite fields.

\begin{corollary}
Consider the field extension \(\mathbb{F}_q \subsetneq \mathbb{F}_{q^{m}}\), where \(q\) is a prime power and \( m \) is a composite number. Let \( t \) be the smallest prime divisor of \( m \), and let \( n \) be an integer with \(t \le n < m\). Then there exist unmatchable subspaces \(A,B \subseteq \mathbb{F}_{q^{m}}\) over $\mathbb{F}_q$ with \(\dim A=\dim B=n\) and \(1\notin B\) if and only if there exists a positive divisor \(d\) of \(m\) such that \( d \leq n \) and \(d\nmid (n+1)\).
\end{corollary}

\begin{proof}
This follows immediately from Theorem~\ref{BoundaryLin}, since the intermediate fields of the extension \(\mathbb{F}_q \subseteq \mathbb{F}_{q^{m}}\) are precisely the fields of the form \( \mathbb{F}_{q^{d}} \), where \( d \) is a positive divisor of \( m \), and \( [\mathbb{F}_{q^d}:\mathbb{F}_q] = d \).
\end{proof}

\section{Future directions}\label{future}

\begin{enumerate}
    \item We believe that a non-abelian version of Theorem~\ref{Structure groups} is plausible. The natural question is: what replaces the nearly periodic decomposition in this setting, and can obstructions be described in terms of coset-like structures or subgroup actions?
    
    \item A natural next step would be to investigate pairs \( (A, B) \) which are partially matchable, that is, matchable up to a prescribed defect. Some useful guidance may be provided by the ``defect versions" of Hall’s marriage theorem and Rado’s theorem (see \cite{Mirsky}).

    \item From an algorithmic perspective, given subsets (or subspaces) $A$ and $B$, can one efficiently detect whether they admit a nearly periodic decomposition (or subfield-translate decomposition in the linear case)? What is the computational complexity of deciding unmatchability in these contexts? 
\end{enumerate}

\section*{Acknowledgment}
The authors are grateful to the anonymous referee, who provided many helpful suggestions for improving the paper.

\bigskip

\noindent \textbf{Declarations} 

\medskip

\noindent \textbf{\small Conflict of interest:} {\small The authors declare that they have no conflict of interest.}\\
\textbf{\small Data availability statement:} {\small Data sharing is not applicable to this article as no datasets were generated or analyzed during the current study.}


\begin{thebibliography}{10}

\bibitem{Aliabadi 0}
M. Aliabadi and K. Filom, Results and questions on matchings in abelian groups and vector subspaces of fields, \textit{J. Algebra} 598 (2022) 85--104.

\bibitem{Aliabadi 4}
M. Aliabadi, J. Kinseth, C. Kunz, H. Serdarevic and C. Willis, Conditions for matchability in groups and field extensions,
  \emph{Linear Multilinear Algebra} 71 (7) (2023) 1182--1197.

\bibitem{Aliabadi 5}
M. Aliabadi, J. Losonczy, Characterization of matchable sets and subspaces via Dyson transforms, \emph{J.\ Algebra} 706 (2026) 221--242.

\bibitem{Aliabadi 3}
M. Aliabadi and S. Zerbib, Matchings in matroids over abelian groups, \textit{J. Algebraic Combin.} 59 (2024) 761--785.

\bibitem{Eliahou 1}
 S. Eliahou and C. Lecouvey, Matchings in arbitrary groups, \textit{Adv. in Appl. Math.} 40 (2008) 219--224.
 
\bibitem{Eliahou 2}
 S. Eliahou and C. Lecouvey, Matching subspaces in a field extension, \textit{J. Algebra} 324 (2010) 3420--3430.

\bibitem{Losonczy 1}
C. K. Fan and J. Losonczy, Matchings and canonical forms for symmetric tensors, \textit {Adv. Math.} 117 (2) (1996) 228--238.

\bibitem{Hamidoune}
 Y. O. Hamidoune, Counting certain pairings in arbitrary groups, \textit{Combin. Probab. Comput.} 20 (6) (2011) 855--865.
    
\bibitem{Losonczy 2}
 J. Losonczy, On matchings in groups, \textit{Adv. in Appl. Math.} 20 (3) (1998) 385--391.
  
\bibitem{McCarthy}
 P. J. McCarthy, \textit{Algebraic Extensions of Fields}, Dover Publications, New York, 1991.  
  
\bibitem{Mirsky}
L.\ Mirsky, \textit{Transversal Theory}, Academic Press, New York, 1971. 
 
\bibitem{Wakeford}
E. K. Wakeford, On canonical forms, \textit{Proc. Lond. Math. Soc.} (2) 18 (1920) 403--410.

\end{thebibliography}
\end{document}